\documentclass[twoside, final]{amsart} %% to see the labels of equations, sections, theorems, etc., 
\usepackage{amsmath, amsfonts, amsthm, amssymb}
\usepackage{exscale}
\usepackage{cite}
\usepackage{epsfig}
\usepackage{amscd}
\usepackage{graphics}
\usepackage{color}
\usepackage{dsfont}

\usepackage[notref, notcite]{showkeys}

\renewcommand{\geq}{\geqslant}
\renewcommand{\leq}{\leqslant}

\newtheorem{theorem}{Theorem}

\newtheorem{proposition}{Proposition}[section]

\newtheorem{lemma}[proposition]{Lemma}
\newtheorem*{main-theorem}{Main Theorem}
\newtheorem*{theorem*}{Theorem}

\theoremstyle{definition}

\newtheorem{remark}[proposition]{Remark}

\newtheorem*{remark*}{Remark}

\numberwithin{equation}{section}

\def\phi{\varphi}

\def\p{\partial}
\def\ZZ{{\mathbb Z}}
\def\NN{{\mathbb N}}
\def\reals{{\mathbb R}}
\def\cx{{\mathbb C}}

\def\Ci{{\mathcal C}^\infty}
\def\Re{\,\mathrm{Re}\,}
\def\Im{\,\mathrm{Im}\,}

\def\sgn{\mathrm{sgn}\,}

\def\supp{\mathrm{supp}\,}

\def\id{\,\mathrm{id}\,}

\def\O{{\mathcal O}}

\def\s{{\mathcal S}}

\def\Op{\mathrm{Op}\,}

\def\csh{{\left( h/\tilde{h} \right)}}
\def\phi{\varphi}

\def\be{\begin{eqnarray*}}
\def\ee{\end{eqnarray*}}
\def\ben{\begin{eqnarray}}
\def\een{\end{eqnarray}}

\def\lll{\left\langle}
\def\rrr{\right\rangle}

\def\L2R{L_{\text{Rest}}^2}

\def\11{\mathds{1}}

\def\RR{\mathbb{R}}
\def\tpsi{\tilde{\psi}}

\def\L2c{L^2_{\text{comp}}}

\def\tg{\tilde{g}}

\def\th{\tilde{h}}

\def\tDelta{\widetilde{\Delta}}

\def\tP{\widetilde{P}}
\def\tu{\tilde{u}}

\def\tR{\tilde{R}}

\def\11{\mathbb{1}}
\def\Vol{\text{Vol}}
\def\uhi{u_{\text{hi}}}
\def\ulo{u_{\text{lo}}}
\def\tQ{\widetilde{Q}}

\def\tL{\tilde{L}}

\newcommand{\Lap}{\Delta}
\newcommand{\abs}[1]{{\left\lvert{#1}\right\rvert}}
\newcommand{\norm}[1]{{\left\lVert{#1}\right\rVert}}
\newcommand{\ang}[1]{{\left\langle{#1}\right\rangle}}
\newcommand{\pa}{{\partial}}
\newcommand{\ep}{{\epsilon}}
\newcommand{\hamvf}{{\textsf{H}}}

\newcommand{\B}{\mathcal{B}}

\begin{document}

\title[Local Smoothing with Loss]{Local smoothing for the
  Schr\"odinger equation with a prescribed loss}

\author[Christianson]{Hans~Christianson}
\email{hans@math.unc.edu}
\address{Department of Mathematics, UNC-Chapel Hill \\ CB\#3250
Phillips Hall \\ Chapel Hill, NC 27599}

\author[Wunsch]{Jared~Wunsch}
\email{jwunsch@\allowbreak math.\allowbreak northwestern.\allowbreak edu}

\subjclass[2000]{}

\begin{abstract}
We consider a family of surfaces of revolution, each with a single periodic
geodesic which is degenerately unstable.  We prove a local smoothing estimate
for solutions to the linear Schr\"odinger equation with a loss that
depends on the degeneracy, and we construct explicit examples to show
our estimate is saturated on a weak semiclassical time scale.  As a
byproduct of our proof, we obtain a cutoff resolvent estimate with a
sharp polynomial loss.

\end{abstract}

\maketitle

\section{Introduction}
\label{S:intro}
Local smoothing for the linear Schr\"odinger equation has a long and
rich history.  First observed by Kato \cite{Kato} for the KdV equation
and later studied by Constantin-Saut \cite{ConSau}, Sj\"olin
\cite{Sjolin}, Vega \cite{Vega}, and Kato-Yajima \cite{KaYa-smooth} for the Schr\"odinger equation, the
local smoothing estimate expresses that, on average in time and
locally in space, solutions to a linear homogeneous dispersive
equation gain some regularity compared to the initial data.  Since
dispersive equations are time-reversible, the propagator at time $t$
preserves the initial energy, but local smoothing shows there is
greater regularity if we also integrate in time.

For the linear Schr\"odinger equation in $\reals^n$, it is well known
(see, for example, \cite{Tao-book})
that for any $\epsilon>0$, there exists $C>0$ such that one has the estimate
\[
\int_0^T \| \lll x \rrr^{-1/2 - \epsilon} e^{it \Delta} u_0
\|_{H^{1/2}}^2 dt \leq C \| u_0 \|_{L^2}^2.
\]
There are several ways to prove this estimate, the simplest of which
is a positive commutator method (see below), although estimates on the
cutoff free resolvent also imply this estimate (this argument seems to
have its origin in the work of Kato \cite{Ka-wos}; 
see also, for example,
\cite{Bur-sm,Chr-disp-1,Chr-sch2}).  
However, for the Schr\"odinger equation on a non-compact manifold, the
situation is not so simple.  A remarkable result of Doi \cite{Doi}
states that one has the sharp $H^{1/2}$ local smoothing effect on an
asymptotically Euclidean manifold if and
only if the geodesic flow is non-trapping.  That is, the presence of
geodesics which do not ``escape to infinity'' cause a loss in how much
regularity the solution can gain.  This has been generalized to
boundary value problems in \cite{Bur-sm}.  In the case of sufficiently
``thin'' hyperbolic trapped sets it has been demonstrated 
in \cite{Bur-sm,Chr-disp-1,Chr-sch2,Dat-sm} that one has only a
``trivial'' loss of $\epsilon>0$ derivatives.\footnote{In fact, with
  some care in definitions, the loss is only logarithmic.}  These examples include
Ikawa's examples \cite{Ika-2-wd,Ika-3-wd,Bur-sm}, a single periodic
hyperbolic geodesic (with or without boundary reflections)
\cite{Chr-disp-1}, very general fractal trapped sets without
boundary \cite{NoZw-res,Chr-sch2,Dat-sm}, and normally hyperbolic
trapped sets \cite{WuZw10}.  That is, in all of these cases,
the authors prove that for any $\epsilon>0$, there exists a constant
$C>0$ such that
\[
\int_0^T \| \lll x \rrr^{-1/2 - \epsilon} e^{it \Delta} u_0
\|_{H^{1/2-\epsilon}}^2 dt \leq C \| u_0 \|_{L^2}^2.
\]
In this case, we call the loss due to trapping ``trivial''.

To contrast, if a manifold admits an elliptic trapped set, the
existence of resonances converging exponentially to the real axis and
the existence of 
infinite order quasimodes
prevents polynomial gain in regularity.

The
purpose of this note is to exhibit a class of manifolds with only one
periodic geodesic which is weakly hyperbolic, and prove a (sharp)
local smoothing effect with loss that lies somewhere between the
complete loss of an elliptic trapped set and the trivial loss of a
strictly hyperbolic trapped set.

We consider the manifold $X = \reals_x \times \reals_\theta / 2 \pi
\ZZ$, equipped with a metric of the form
\[
ds^2 = d x^2 + A^2(x) d \theta^2,
\]
where $A \in \Ci$ is a smooth function, $A \geq \epsilon>0.$ 
From this metric, we get the volume form
\[
d \Vol = A(x) dx d \theta,
\]
and the Laplace-Beltrami operator acting on $0$-forms
\[
\Delta f = (\partial_x^2 + A^{-2} \partial_\theta^2 + A^{-1}
A' \partial_x) f.
\]
We observe that we can conjugate
$\Delta$ by an isometry of metric spaces and separate variables so
that spectral analysis of $\Delta$ is equivalent to a one-variable
semiclassical problem with potential.  That is, let $T : L^2(X, d
\Vol) \to L^2(X, dx d \theta)$ be the isometry given by
\[
Tu(x, \theta) = A^{1/2}(x) u(x, \theta).
\]
Then $\tDelta = T \Delta T^{-1}$ is essentially self-adjoint on $L^2 (
X, dx d \theta)$ with mild assumptions on $A$ (for example in this
paper $X$ has two ends which are short range perturbations of $\reals^2$).  A simple calculation gives
\[
-\tDelta f = (- \partial_x^2 - A^{-2}(x) \partial_\theta^2 + V_1(x) )
f,
\]
where the potential
\[
V_1(x) = \frac{1}{2} A'' A^{-1} - \frac{1}{4} (A')^2 A^{-2}.
\]

If we now separate variables and write $\psi(x, \theta) = \sum_k
\phi_k(x) e^{ik \theta}$, we see that
\[
(-\tDelta- \lambda^2) \psi = \sum_k e^{ik \theta} (P_k -\lambda^2)\phi_k(x),
\]
where
\[
(P_k -\lambda^2) \phi_k(x) = (-\frac{d^2}{dx^2} + k^2 A^{-2}(x) + V_1(x) - \lambda^2)
\phi_k(x).
\]
Setting $h = k^{-1}$, we have the semiclassical operator
\[
P(z,h) \phi(x) = (-h^2 \frac{d^2}{dx^2} + V(x) -z) \phi(x),
\]
where the potential is
\[
V(x) = A^{-2}(x) + h^2 V_1(x)
\]
and the spectral parameter is $z = h^2 \lambda^2$.  

In this paper, we are primarily interested in the case $A(x) = (1 + x^{2m})^{1/2m}$, $m
\in \ZZ_+$.  If $m \geq 2$, then $X$ is asymptotically Euclidean
(with two ends), and the subpotential $h^2 V_1(x)$ is seen to be
lower order in both the semiclassical and the scattering sense.  If $m
= 1$, a trivial modification must be made to make the metric a
short-range perturbation, but we completely ignore this issue here.  The
point is that for $m \geq 2$, the principal part of the potential $V(x)$
is $A^{-2}(x)$ which has a degenerate maximum at $x = 0$.  The
corresponding periodic geodesic $\gamma \subset X$ is {\it weakly} hyperbolic in
the sense that it is unstable and isolated, but degenerate (see Figure
\ref{fig:fig1}).

\begin{figure}
\hfill
\centerline{\input{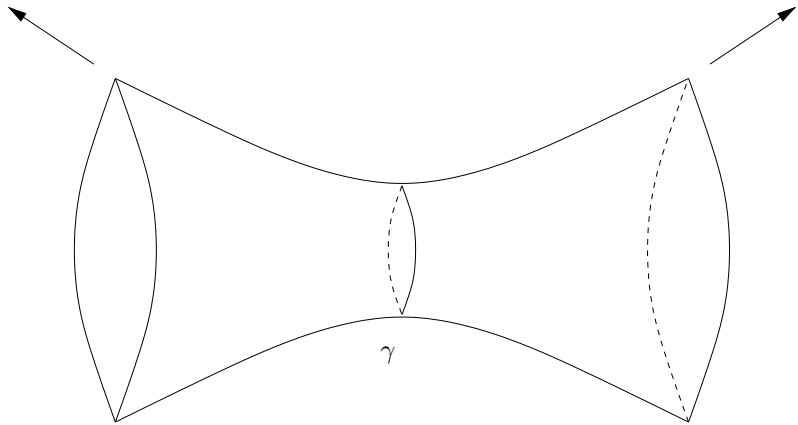}}
\caption{\label{fig:fig1} A piece of the manifold $X$ and the periodic
  geodesic $\gamma$.  The Gaussian curvature $K = -A''/A =
  -(2m-1)x^{2m-2} (1 + x^{2m})^{-2}$ vanishes to order $2m-2$ at $x =
  0$ and is asymptotically flat as $|x| \to \infty$  }
\hfill
\end{figure}

Our main result is the following theorem, which says that for every $m \geq 2$, there is still some
local smoothing, but with a polynomial loss depending on $m$.  

\begin{theorem}[Local Smoothing]
\label{T:smoothing}
Suppose $X$ is as above for $m \geq 2$, and assume $u$ solves
\[
\begin{cases} (D_t -\Delta ) u = 0 \text{ in } \reals \times X , \\
u|_{t=0} = u_0 \in H^{s}
\end{cases}
\]
for some $s \geq m/(m+1)$.  
Then for any $T<\infty$, there exists a constant
$C>0$ such that 
\[
\int_0^T  \| \lll x \rrr^{-3/2}  u \|_{H^1(X)}^2 \, dt \leq C (\| \lll D_\theta
\rrr^{m/(m+1)} u_0 \|_{L^2}^2 + \| \lll D_x \rrr^{1/2} u_0 \|_{L^2}^2).
\]

\end{theorem}

\begin{remark}
Observe that there is no polynomial local smoothing effect 
in the limit $m \to \infty$.  In Theorem \ref{T:sharp} below, we show
Theorem \ref{T:smoothing} is sharp, and that in fact the estimate is
saturated on a weak semiclassical time scale.

\end{remark}

We are also able to prove, using the same techniques, a polynomial
bound on the resolvent of the Laplacian in the same geometric
setting.    We now assume for simplicity that our surface of
revolution is Euclidean at infinity, i.e.\ that $A(x) =
x$ for $\abs{x}\gg 0.$  (More generally we could just require
merely dilation analyticity at infinity; this would allow us to
include asymptotically conic spaces as treated in \cite{WZ}.)

We denote by $R(\lambda)$
the resolvent on $X$
\[
R(\lambda) = (-\Delta_g - \lambda^2)^{-1},
\]
where it exists.  If we take $\Im \lambda < 0$ as our physical sheet,
then, since $X$ is Euclidean near infinity, there is a meromorphic
continuation of $\chi R(\lambda) \chi$ to the logarithmic covering
space, for any $\chi \in \Ci_c(X)$ (see, e.g., \cite{SjZw}).
In particular, by choosing an
appropriate branch cut, $\chi R(\lambda) \chi$ continues
meromorphically to $\{ \lambda \in \reals, \lambda \gg 0 \}$.  

% \jw{Moved the statement of resolvent est.\ up to here.  Added
%  hypothesis that we're Euclidean at infinity.  I presume we could get
%  away with less, but if so should state carefully.}
\begin{theorem}
\label{T:resolvent}
Fix $m \geq 2$.  For any $\chi \in \Ci_c(X)$, there exists a constant
$C= C_{m, \chi} >0$ such that
for $\lambda \gg 0$, 
\[
\| \chi R(\lambda-i0) \chi \|_{L^2 \to L^2} \leq C \lambda^{-2/(m+1)}.
\]
Moreover, this estimate is sharp, in the sense that no better
polynomial rate of decay holds true.

\end{theorem}

\begin{remark}
The estimate in this theorem represents a loss over the non-trapping
case, when generally $\lambda^{-1}$ order bounds are known.  In the
non-degenerate hyperbolic trapping case ($m=1$), most known estimates
are of the order $\lambda^{-1} \log (\lambda)$, and in the elliptic
trapping cases, generally one expects at best exponential bounds.
Hence Theorem \ref{T:resolvent} represents a family of estimates with
a sharp polynomial loss.  To the best of our knowledge, no other such
examples are known.
\end{remark}

\subsection*{Acknowledgements}
The authors would like to thank J.~Marzuola, R.~Melrose, J.~Metcalfe,
and M.~Zworski for helpful conversations.  We are especially grateful
to N.~Burq for suggesting the model problem treated here as one which
might exhibit a finite loss of local smoothing.  We would also like to
thank Kiril Datchev for suggesting we write up the resolvent estimate
with loss.  Finally, we would like to thank the anonymous referee
whose careful reading of this manuscript has greatly helped improve
the presentation.

\section{Positive commutators}

\subsection{The smoothing estimate on Euclidean space}
In this section we write out the standard positive commutator proof of local smoothing for
the Schr\"odinger equation in polar coordinates.  We then try to mimic
the proof in the case of degenerate hyperbolic orbits ($m \geq 2$ above)
to see where the proof fails.

In polar coordinates, the homogeneous Schr\"odinger equation on
$\RR_t \times \RR^2$ is
\[
\begin{cases}
(D_t - \partial_r^2 - r^{-1} \partial_r - r^{-2} \partial_\theta^2)u =
0, \\
u|_{t = 0} = u_0;
\end{cases}
\]
we will of course write
$$
\Lap = \pa_r^2 + r^{-1} \pa_r +r^{-2} \pa_\theta^2.
$$
We recall that in polar coordinates the radial, or scaling, vector
field is $x
\cdot \partial_x = r \partial_r$.  By scaling, we immediately compute
$$
[r\pa_r, \Lap] =2 \Lap;
$$
however, as $r \pa_r$ is not a bounded map between Sobolev spaces, we
change the weight and
employ the commutant $B = r \lll r
\rrr^{-1} \partial_r$. 
 The function $a(r) = r \lll r \rrr^{-1}$ is
non-negative and bounded, and satisfies $a'(r) = \lll r \rrr^{-3}$.
Thus, we compute
\begin{align}
[B, \Lap] & = 
2 a' \partial_r^2 + (a''+ a' r^{-1} + a r^{-2}) \partial_r + 2 a
r^{-3} \partial_\theta^2 \label{foobar}\\
& = 2 \lll r \rrr^{-3} \partial_r^2 + 2 \lll r \rrr^{-1}
r^{-2} \partial_\theta^2+O(r^{-1} \ang{r}^{-1}) \pa_r. \notag
\end{align}
Using the Schr\"odinger equation, we write
\begin{align*}
  0  = & 2 i \Im \int_0^T \lll B (D_t -\Lap)u, u \rrr dt \\
 = & \int_0^T \lll B (D_t - \partial_r^2 -
  r^{-1} \partial_r - r^{-2} \partial_\theta^2)u, u \rrr \\
& - \lll u, B (D_t - \partial_r^2 -
  r^{-1} \partial_r - r^{-2} \partial_\theta^2)u \rrr dt \\
 = & \int_0^T \lll [B, ( - \partial_r^2 -
  r^{-1} \partial_r - r^{-2} \partial_\theta^2)]u, u \rrr dt + i
  \left. \lll
  B u , u \rrr \right|_0^T.
\end{align*}
The last term is bounded using energy estimates by
\[
\left| \left. \lll
  B u , u \rrr \right|_0^T \right| \leq \| u_0 \|_{H^{1/2}}^2.
\] 
Rearranging, we thus obtain
$$
\int_0^T \lll [B, \Lap]u, u \rrr dt\leq C_T \norm{u_0}_{H^{1/2}}^2.
$$
Employing \eqref{foobar} and integrating by parts thus yields
$$
\int_0^T \norm{\ang{r}^{-3/2}\pa_r u}^2 + \norm{\ang{r}^{-1/2} r^{-1}\pa_\theta
  u}^2\, dt
\leq C_T \norm{u_0}_{H^{1/2}}^2.
$$
where we have absorbed on the right the term involving $\int_0^T
\ang{\pa_r u,u} \, dt$ as well as the similar error terms from
commuting $\pa_r$ with a multiplier. This is the local smoothing
estimate on the manifold $\reals^2$.

\subsection{Degenerate hyperbolic trapping}
In this section, we prove our main local smoothing estimate.

% we argue using two estimates on
% localization of $L^2$ solutions in phase space.  That is, we first
% suppose $u \in L^2$ satisfies 
% \[
%  Qu = \O(h^{2m/(m+1)-\epsilon}) \| u \|.
% \]
% Then $u$ is localized in phase space to where the principal symbol of
% $Q$, $q = \xi^2 + x^{2m}$ is bounded by $\O(h^{2m/(m+1)-\epsilon})$.  But then
% \[
% \text{Area} (\xi^2 + x^{2m} = \O(h^{2m/(m+1)-\epsilon})) = \O(h),
% \]
% which is a violation of the uncertainty principle if we can show that
% $u$ is non-spreading.

% For that, assume now that $u$ satisfies $Q u = \O(h^{2m/(m+1)}) \| u
% \|$, and observe if we rescale
% \[
% \tilde{Q} u = \O(1) \| u \|,
% \]
% where
% \[
% \tilde{Q} = -h^{2-2m/(m+1)}\partial_x^2 + h^{-2m/(m+1)}x^{2m},
% \]
% then $\tilde{Q}$ has $h$-principal symbol
% \[
% \tilde{q} = (h^{-m/2(m+1)} \xi)^2 + (h^{-1/(m+1)} x)^{2m}.
% \]

Let us begin by reproducing the positive commutator computation in the
previous section for the degenerate case.  Let $A(x) = (1 +
x^{2m})^{1/2m}$, the metric $ds^2 = dx^2 + A^2 d \theta^2$ as before,
and conjugate the Laplacian to Euclidean space:
\[
-\tDelta f = (- \partial_x^2 - A^{-2}(x) \partial_\theta^2 + V_1(x) )
f,
\]
where the potential
\[
V_1(x) = \frac{1}{2} A'' A^{-1} - \frac{1}{4} (A')^2 A^{-2}.
\]

The following proposition is the statement of local smoothing for the
conjugated equation, and evidently implies Theorem \ref{T:smoothing}
by conjugating back.
\begin{proposition}
\label{P:smoothing}
Suppose $m \geq 2$ and $u$ solves 
\begin{equation}
\label{E:tDelta-Sch}
\begin{cases} (D_t -\tDelta ) u = 0, \\
u(0,x, \theta) = u_0.
\end{cases}
\end{equation}
Then for any $T<\infty$ there exists a constant
$C>0$ such that 
\begin{align*}
\int_0^T & ( \| \lll x \rrr^{-1} \partial_x u \|_{L^2}^2 + \| \lll x
\rrr^{-3/2} \partial_\theta u \|_{L^2}^2 ) \, dt \\
& \leq C (\| \lll D_\theta
\rrr^{m/(m+1) } u_0 \|_{L^2}^2 + \| \lll D_x \rrr^{1/2} u_0 \|_{L^2}^2).
\end{align*}
\end{proposition}

\subsection{Proof of Proposition \ref{P:smoothing}}

Let us summarize briefly the strategy of the proof.  Using a positive
commutator argument similar to the previous section, we prove local
smoothing except at the periodic orbit $\gamma = \{x = 0 \}$.
Moreover, solutions to \eqref{E:tDelta-Sch} exhibit perfect local
smoothing in the $x$ direction and only lose smoothing in the
directions tangential to $\gamma$ (that is, only in the $\theta$
direction).  Thus it suffices to prove local smoothing with a loss for
$\theta$ derivatives, in a neighbourhood of $x = 0$.  We separate
variables in the $\theta$ direction (Fourier series decomposition) and
prove estimates uniform in each Fourier mode.  To do this, we further
decompose, say, the $k$th Fourier mode into a low-frequency part where $|k| \leq
|D_x|$ and a high-frequency part where $|D_x| \leq |k|$.  The low
frequency part is estimated using the positive commutator technique
modulo a term which is localized to high-frequencies, so it suffices
to estimate a solution cut off to high frequencies.  For this, we
introduce a semiclassical rescaling, and 
reduce the estimate to a cutoff semiclassical resolvent estimate,
which implies local smoothing via \cite[Theorem 1]{Chr-sch2}.

\noindent {\bf Step 1: Positive commutators and the estimate away from $x=0$.}

If $B = \arctan(x) \partial_x$, we have
\[
[\tDelta, B] = 2 \lll x \rrr^{-2} \partial_x^2 - 2 x \lll x
\rrr^{-4} \partial_x + 2 A' A^{-3} \arctan(x) \partial_\theta^2 + V_1'
\arctan(x).
\]
Now 
\[
iB- (iB)^* = i[\arctan(x) , \partial_x]
\]
is $L^2$ bounded, so 
\begin{align*}
0  =&  \int_0^T \int u \overline{ ( \arctan(x) D_x (D_t - \tDelta) u)} dx
d \theta \, dt \\
 =&  \int_0^T \int \arctan(x) D_x u \overline{ ( (D_t - \tDelta) u)} dx
d \theta \, dt
\\
& + \int_0^T \int (iB-(iB)^*) u \overline{ ( (D_t - \tDelta) u)} dx
d \theta \, dt \\
 = & i \lll \arctan(x) D_x  u , u \rrr|_0^T + \int_0^T \lll (D_t -
\tDelta) i^{-1} B u, u \rrr \, dt.
\end{align*}
Hence, using the notation $P = D_t - \tDelta$,
\begin{align*}
0 & = 2 i \Im \int_0^T \lll i^{-1} B P u, u \rrr \, dt \\
& = \int_0^T \lll i^{-1} B P u, u \rrr \, dt - \int_0^T \lll u, i^{-1} B
P u \rrr \, dt \\
& = \int_0^T \lll [i^{-1}B,P] u, u \rrr \, dt - i \lll \arctan(x) D_x u , 
u \rrr|_0^T ,
\end{align*}
or
\[
 \int_0^T \lll [B,-\tDelta] u, u \rrr \, dt = - \lll \arctan(x) D_x u , 
u \rrr|_0^T,
\]
since $B$ does not depend on $t$.  By writing $\partial_x = \lll D_x
\rrr^{1/2} \lll D_x \rrr^{-1/2} \partial_x$, and using energy
estimates, we can control the right hand side by $\| u_0
\|_{H^{1/2}}^2$.  The left hand side is computed as above:
\begin{align*}
 \int_0^T & \lll [B,-\tDelta] u, u \rrr \, dt \\
 = & \int_0^T \Big\langle ( 2 \lll x \rrr^{-2} \partial_x^2 - 2 x \lll x
\rrr^{-4} \partial_x + 2 A' A^{-3} \arctan(x) \partial_\theta^2 \\
& + V_1'
\arctan(x)) u, u \Big\rangle \, dt.
\end{align*}
Using the energy estimates, 
\begin{align}
\Big| \int_0^T & \lll ( - 2 x \lll x
\rrr^{-4} \partial_x + V_1'
\arctan(x)) u, u \rrr \, dt \Big| \leq C T\sup_{0 \leq t \leq T } \|
u(t) \|_{H^{1/2}}^2 \notag \\
& \leq C_T \| u_0 \|_{H^{1/2}}^2. \label{E:energy-RHS}
\end{align}

Integrating by parts in $x$ and $\theta$ and adding
the lower order terms into the right hand side as in
\eqref{E:energy-RHS} yields the estimate
\begin{equation*}
\int_0^T (\| \lll x \rrr ^{-1} \partial_x u \|_{L^2}^2 + \|\sqrt{A' A^{-3} \arctan(x) } \partial_\theta u \|_{L^2}^2 ) \, dt \leq C \| u_0
\|_{H^{1/2}}^2.
\end{equation*}
We observe that 
\[
A' A^{-3} \arctan(x) = \arctan(x) x^{2m-1} (1 + x^{2m})^{-1/m -1}
\]
is even, non-negative, bounded below by $C|x|^{2m}$ for $|x| \leq 1$ and
$C'|x|^{-3}$ for $|x| \geq 1$.  Hence 
\[
|x|^{2m} \lll x \rrr^{-2m-3} \leq CA' A^{-3} \arctan(x) ,
\]
and hence, 
\[
\lll |x|^{2m} \lll x \rrr^{-2m-3} \partial_\theta u, \partial_\theta u
\rrr \leq C \lll A' A^{-3} \arctan(x) \partial_\theta
u, \partial_\theta u \rrr
\]
plus terms which can be absorbed into the energy, 
so up to lower order terms, 
\[
\| |x|^{m} \lll x \rrr^{-m-3/2} \partial_\theta u \| \leq C \|
\sqrt{A' A^{-3} \arctan(x) } \partial_\theta u \|.
\]
Hence we have the estimate
\begin{equation}
\label{E:est-away-0}
\int_0^T (\| \lll x \rrr ^{-1} \partial_x u \|_{L^2}^2 + \| |x|^{m}
\lll x \rrr^{-m-3/2} \partial_\theta u \|_{L^2}^2 ) \, dt \leq C \| u_0
\|_{H^{1/2}}^2.
\end{equation}

\noindent {\bf Step 2: Low frequency estimate.}

The estimate \eqref{E:est-away-0} shows we have perfect local
smoothing away from the periodic geodesic at $x=0$, and moreover shows
we have perfect local smoothing in the $x$ direction.  That is, the
only loss is in the direction tangential to the periodic geodesic,
which is expected since a point in $T^*X$ which is transversal to the
periodic geodesic will flow out to infinity, and only the tangential
directions stay localized for long times.  In this subsection we show
how to get an estimate in the tangential directions with a loss.

Let us decompose
\[
u(t,x,\theta) = \sum_k e^{ik \theta} u_k(t, x),
\]
and
\[
u_0(x, \theta) = \sum_k e^{ik \theta} u_{0,k}(x).
\]
By orthogonality it suffices to prove Proposition \ref{P:smoothing}
for each mode.  
Observe the zero mode $u_0(t,x)$ satisfies
\[
\begin{cases}
(D_t -\tDelta) u_0(t,x) = 0, \\
u_0(0,x) = u_{0,0}(x),
\end{cases}
\]
and $\p_\theta u_0(t,x) = 0$, so that, from Step 1, we have
\begin{align*}
\int_0^T & ( \| \lll x \rrr^{-1} \partial_x u_0 \|_{L^2}^2 + \| \lll x
\rrr^{-3/2} \partial_\theta u_0 \|_{L^2}^2 ) \, dt = \int_0^T  \| \lll
x \rrr^{-1} \partial_x u_0 \|_{L^2}^2  \, dt \\
& \leq \int_0^T (\| \lll x \rrr ^{-1} \partial_x u_0 \|_{L^2}^2 + \| |x|^{m}
\lll x \rrr^{-m-3/2} \partial_\theta u_0 \|_{L^2}^2 ) \, dt \\
& \leq C \| u_{0,0}
\|_{H^{1/2}}^2\\
& \leq C (\| \lll D_\theta
\rrr^{m/(m+1) } u_0 \|_{L^2}^2 + \| \lll D_x \rrr^{1/2} u_0 \|_{L^2}^2).
\end{align*}
To prove Proposition \ref{P:smoothing} for the nonzero modes, we 
show
\[
\int_0^T \| \chi(x) k u_k \|_{L^2(\reals)}^2 \, dt \leq C (\| \lll k \rrr^{m/(m+1)}
u_{0,k} \|_{L^2}^2 + \| u_{0,k} \|_{H^{1/2}}^2 )
\]
for some $\chi \in \Ci_c( \reals)$ with $\chi (x) \equiv 1$ near $x =
0$, for $| k | \geq 1$.

For simplicity in exposition, let us drop the $k$ notation for $u$ and
$u_0$, and just observe that now the time-dependent Schr\"odinger
operator depends on $k$:
\[
D_t + P_k = D_t -\tDelta =  D_t + D_x^2 + A^{-2}(x) k^2 + V_1(x) .
\]
The idea is to use $k^{-1}$ as a semiclassical parameter, and
decompose $u$ into a part where $|k| \leq |D_x|$ and a part where $k$
is not controlled by $D_x$.  It turns out that then the part not
controlled by $D_x$ can be handled with a second commutator argument
plus a $T T^*$ argument (see Step 3 below
and \S \ref{SS:ml-res-est-pf-ssect}).  

Let $\psi \in \Ci_c( \reals)$ be an even function satisfying $\psi (r)\equiv 1$ for $| r |
\leq 1$ and $\psi(r) \equiv 0$ for $| r | \geq 2$.  Let
\[
u = \uhi + \ulo,
\]
where
\[
\uhi = \psi(D_x/k) u, \,\,\, \ulo = (1 - \psi(D_x/k)) u.
\]
Observe $\ulo$ satisfies the equation
\[
(D_t + P_k )\ulo = -[P_k, \psi(D_x/k)] u = k \lll x \rrr^{-3} L \tpsi(D_x/k) u,
\]
where $L$ is $L^2$ bounded and $\tpsi \in \Ci_c$ is $1$ on $\supp
\psi$.  If we try to apply the positive commutator argument from the previous
step to $\ulo$, we now have
\begin{align*}
2 i \Im &\int_0^T \lll i^{-1} B k \lll x \rrr^{-3} L \tpsi(D_x/k) u,
\ulo \rrr \, dt \\
  = & 2 i
\Im \int_0^T \lll i^{-1} B (D_t + P_k) \ulo, \ulo \rrr \, dt \\ 
 = & \int_0^T \lll i^{-1} B (D_t + P_k )\ulo, \ulo \rrr \, dt - \int_0^T \lll \ulo, i^{-1} B
(D_t + P_k) \ulo \rrr \, dt \\
& -  \int_0^T \int (iB-(iB)^*) \ulo \overline{  (D_t + P_k) \ulo } dx
d \theta \, dt  \\
 = & \int_0^T \lll [i^{-1}B,P_k] u, u \rrr \, dt - i \lll \arctan(x) D_x \ulo , 
\ulo \rrr|_0^T \\
& - i \int_0^T \lll \lll x \rrr^{-2} \ulo, k \lll x
\rrr^{-3} L \tpsi(D_x/k) u \rrr \, dt,
\end{align*}
or
\begin{align*}
& \left| \int_0^T  \lll [i^{-1}B,P_k] u, u \rrr \, dt \right|  \\ &
\quad \leq   C \Bigg(
  \left| \int_0^T \lll i^{-1} B k \lll x \rrr^{-3} L \tpsi(D_x/k) u,
\ulo \rrr \, dt \right| \\
& \quad \quad + \left| \lll \arctan(x) D_x \ulo , 
\ulo \rrr|_0^T\right|  + \left|  \int_0^T \lll \lll x \rrr^{-2} \ulo, k \lll x
\rrr^{-3} L \tpsi(D_x/k) u \rrr \, dt \right| \Bigg) \\
& \quad \leq  C \Big(  \| k \lll x \rrr^{-3/2} \tpsi(D_x/k) u \|^2 + \| \lll
x \rrr^{-3/2} D_x u \|^2 + \| u_0 \|_{H^{1/2}}^2    \Big),
\end{align*}
where we have again used the energy estimate where appropriate.  
The right hand side is now
controlled by 
\eqref{E:est-away-0} except near $x=0$.  What we have gained is the
cutoff in frequency $\tpsi(D_x/k)$.  Clearly if we can show for any
$\chi \in \Ci_c$, $\chi \equiv 1$ near $x=0$, 
\[
\int_0^T \| \chi k \tpsi(D_x/k) u \|_{L^2}^2\, dt \leq C \| k^{m/(m+1)} u_0 \|_{L^2}^2
\]
we can control the remaining term from the estimate on $\ulo$ as well
as the estimate of $\uhi$.

\noindent {\bf Step 3: The high frequency estimate.}

Let us now try to estimate $\uhi$ near $x=0$, or more generally a
solution to $(D_t + P_k) u =0$ microlocalized near $(0,0)$.  
For some $0 \leq r \leq 1/2$ to be determined, let $F(t)$ be defined by
\[
F(t) g = \chi(x) \psi(D_x/k) k^r e^{-itP_k} g,
\]
where $e^{-itP_k}$ is the free propagator.  Our goal is to determine for
what values of $r$ we have a mapping $F: L^2_x \to L^2([0,T]) L^2_x$,
since then 
\begin{equation}
\label{E:l-sm-r}
\| k^{1-r} F(t) u_0 \|_{L^2([0,T]);L^2} \leq C \| k^{1-r} u_0 \|_{L^2}
\end{equation}
is the desired local smoothing estimate.  
We have such a mapping if and only if $F F^* : L^2 L^2 \to L^2 L^2$.
We compute
\[
F F^* f(x,t) = \psi(D_x/k) \chi (x)  k^{2r} \int_0^T e^{i(t-s)P_k} \chi(x)
\psi(D_x/k) f(x,s)ds,
\]
and it suffices to estimate $\| FF^* f \|_{L^2 L^2} \leq C \| f
\|_{L^2 L^2}.$
We write $F F^* f(x,t) = \psi \chi  (v_1 + v_2)$, where
\[
v_1 = k^{2r} \int_0^t e^{i(t-s)P_k} \chi(x)
\psi(D_x/k) f(x,s)ds,
\]
and
\[
v_2 = k^{2r} \int_t^T e^{i(t-s)P_k} \chi(x)
\psi(D_x/k) f(x,s)ds,
\]
so that
\[
(D_t + P_k)v_j = \pm i k^{2r} \chi \psi f,
\]
and it suffices to estimate 
\[
\| \psi \chi v_j \|_{L^2 L^2} \leq C \| f \|_{L^2 L^2}.
\]
Since the Fourier transform in time is an $L^2$ isometry, it suffices
to estimate
\[
\|\psi \chi \hat{v}_j \|_{L^2 L^2} \leq C \| \hat{f} \|_{L^2 L^2},
\]
but this is the same as estimating
\[
\| \psi \chi k^{2r}(\tau \pm i 0+ P_k)^{-1} \chi \psi \|_{L^2_x \to L^2_x} \leq
C.
\]
Let us factor out the $k^2$ in $P_k$ to get the operator
\[
k^{-2r} (\tau \pm i 0+ P_k) = k^{2(1-r)}(-z \pm i 0+ k^{-2}D_x^2  + A^{-2}(x) + k^{-2}
V_1(x))
\]
for $-z = \tau k^{-2}$,
and if we let $h = k^{-1}$, we are left with the task of finding $r$
so that 
\[
\| \psi(hD_x) \chi(x) (-z \pm i 0+ (hD_x)^2 + V)^{-1} \chi(x) \psi(hD_x)
\|_{L^2 \to L^2} \leq C h^{-2(1-r)},
\]
where $V = A^{-2}(x) + h^2 V_1(x)$.  Let $$\tQ =  (hD_x)^2 + V-z.$$
We observe that the cutoff $\psi(hD_x)\chi(x)$ shows we only need to estimate
this for $z$ in a bounded interval near $z = 1$.  Indeed, $\psi \chi$
cuts off to a neighbourhood of $(0,0)$, and $V(0) = 1$, so for $|
z-1|$ sufficiently large, we have elliptic regularity.  The cutoff estimate
on $\tQ$ is the content of the following Proposition, which
is proved in the next subsection.

\begin{proposition}
\label{P:ml-res-est}
Let $\phi \in \Phi^0$ have wavefront set sufficiently close to
$(0,0)$.  Then for each $\epsilon>0$ sufficiently small, there exists
a constant $C>0$ such that
\[
\| \phi (\tQ \pm i 0 )^{-1} \phi \|_{L^2 \to L^2} \leq C h^{-2m/{m+1}}, \,\,\, z
\in [1-\epsilon, 1 + \epsilon].
\]
\end{proposition}

With Proposition \ref{P:ml-res-est} in hand, we observe 
\[
\| \psi \chi k^{2r}(\tau \pm i 0  + P_k)^{-1} \chi \psi \|_{L^2_x \to L^2_x} \leq
C
\]
holds if 
\[
k^{2(r-1)} = k^{-2m/(m+1)},
\]
or 
\[
r = \frac{1}{m+1}.
\]
From \eqref{E:l-sm-r}, this implies Proposition \ref{P:smoothing} (see
also \cite[Theorem 1]{Chr-sch2}).

\subsection{Proof of Proposition \ref{P:ml-res-est}}
\label{SS:ml-res-est-pf-ssect}

The technique of proof is to prove an invertibility estimate
microlocally near $(0,0)$ in Lemma \ref{L:ml-inv} below.  From this, one easily obtains a resolvent
estimate with complex absorbing potential, and then the gluing techniques of
\cite[Proposition 2.2]{Chr-disp-1} imply the Proposition (see also the
recent paper of Datchev-Vasy \cite{DaVa-gluing}).

The proof of the microlocal invertibility estimate proceeds through
several steps.  First, we rescale the principal symbol of $\tQ$ to
introduce a calculus of two parameters.  We then quantize in the
second parameter which eventually will be fixed as a constant in the
problem.  This technique has been used in
\cite{SjZw-mono,SjZw-frac,Chr-NC,Chr-QMNC}.

Our central result to achieve microlocal invertibility is a lower
bound for the resolvent of the semiclassical operator $\tQ,$ whose potential
has a degenerate barrier top.  This result is of independent interest,
and is used to prove the sharp resolvent estimate in Theorem \ref{T:resolvent}.
\begin{lemma}
\label{L:ml-inv}
For $\epsilon>0$ sufficiently small, let $\phi \in \s(T^* \reals)$
have compact support in $\{ |(x,\xi) |\leq \epsilon\}$.  Then there
exists $C_\epsilon>0$ such that 
\begin{equation}
\label{E:ml-inv}
\| \tQ \phi^w u \| \geq C_\epsilon h^{2m/(m+1)} \|
\phi^w u \|, \,\,\, z \in [1-\epsilon, 1 + \epsilon].
\end{equation}
\end{lemma}

\subsection{The two-parameter calculus}
% \jw{Fleshed out this section, 2/22/11.}
% \jw{There's a nasty clash of uses of $\alpha;$ sorry.  Maybe change the multiindices?}

Following Sj\"ostrand-Zworski \cite[\S3.3]{SjZw-frac}, we introduce a
calculus with two parameters, designed to enable symbolic computations
in the $h^{-1/2}$ calculus which would otherwise involve global
considerations rather than a local Moyal product of symbols.  We use a
somewhat more general version of this calculus than in
\cite{SjZw-frac}, involving inhomogenous powers of $h.$

For $\alpha\in [0,1]$ and $\beta\leq 1-\alpha,$ we let
\begin{eqnarray*}
\lefteqn{\s_{\alpha,\beta}^{k,m, \widetilde{m}} \left(T^*(\RR^n) \right):= } \\
& = & \Bigg\{ a \in \Ci \left(\RR^n \times (\RR^n)^* \times (0,1]^2 \right):  \\
 && \quad \quad  \left| \partial_x^\rho \partial_\xi^\gamma a(x, \xi; h, \tilde{h}) \right| 
\leq C_{\rho \gamma}h^{-m}\tilde{h}^{-\widetilde{m}} \left(
  \frac{\tilde{h}}{h} \right)^{\alpha |\rho| + \beta |\gamma|} 
\langle \xi \rangle^{k - |\gamma|} \Bigg\}.
\end{eqnarray*}
Throughout this work we will always assume $\tilde{h} \geq h$.  
We let $\Psi_{\alpha,\beta}^{k, m, \widetilde{m}}$ denote the
corresponding spaces of semiclassical pseudodifferential operators
obtained by Weyl quantization of these symbols. We will sometimes add a
subscript of $h$ or $\tilde{h}$ to indicate which parameter is used in
the quantization; in the absence of such a parameter, the quantization
is assumed to be in $h.$  The class $\s_{\alpha,\beta}$ (with no
superscripts) will denote $\s_{\alpha,\beta}^{0,0,0}$ for brevity.

In \cite{SjZw-frac}, it is observed that in the special case
$\alpha=\beta=1/2,$ the composition in the
calculus can be computed in terms of a symbol product that converges
in the sense that terms improve in $\tilde{h}$ and $\xi$ orders, but
not in $h$ orders (owing to the marginality of the $h^{-1/2}$
calculus, which is what the introduction of the second parameter
$\tilde{h}$ mitigates).  We will restrict our attention in what
follows to a generalization of this marginal case:
$$
\alpha+\beta=1.
$$

By the same arguments employed in \cite{SjZw-frac}, we may easily
verify that the calculus $\Psi_{\alpha,\beta}$ is closed under composition: if
$ a \in   \s^{k,m , {\widetilde m} }_{\alpha,\beta}$ and
$ b \in   \s^{k',m', \widetilde m''}_{\alpha,\beta} $ then 
\[ \Op_h^w (a) \circ \Op_h^w(b) = \Op_h^w (c)
\ \text{ with } \ 
c \in   \s^{k+k',m +m', {\widetilde m}+ {\widetilde m}'  }_{\alpha,\beta}\,.
\]
The presence of the additional parameter $ \tilde h $ allows us to 
conclude that 
\[ c \equiv \sum_{ |\rho | <  M } \frac{1}{\rho !} 
\partial_\xi^\rho a D_x^\rho b \ \mod \s^{ k + k' - M , m  + m ' , {\widetilde m} + {\widetilde m}' 
- M }_{\alpha,\beta} \,, \]
that is, we have a symbolic expansion in powers of $ \tilde h $. 

We also note that a more general version of \cite[Lemma 3.6]{SjZw-frac} holds, giving 
error estimates on remainders:
\begin{lemma}
\label{l:err}
Suppose that 
$ a, b \in \s_{\alpha,\beta}$, 
and that $ c^w = a^w \circ b^w $. 
Then 
\begin{equation}
\label{eq:weylc}  c ( x, \xi) = \sum_{k=0}^N \frac{1}{k!} \left( 
\frac{i h}{2} \sigma ( D_x , D_\xi; D_y , D_\eta) \right)^k a ( x , \xi) 
b( y , \eta) |_{ x = y , \xi = \eta} + e_N ( x, \xi ) \,,
\end{equation}
where for some $ M $
\begin{equation}
\label{eq:new1}
\begin{split}
& | \partial^{\gamma} e_N | \leq C_N h^{N+1}
 \\
& \ \ 
\times \sum_{ \gamma_1 + \gamma_2 = \gamma } 
 \sup_{ 
{{( x, \xi) \in T^* \RR^n }
\atop{ ( y , \eta) \in T^* \RR^n }}} \sup_{
|\rho | \leq M  \,, \rho \in \NN^{4n} }
\left|
\Gamma_{\alpha, \beta, \rho,\gamma }(D)
( \sigma ( D) ) ^{N+1} a ( x , \xi)  
b ( y, \eta ) 
\right| \,,
\end{split} 
\end{equation}
where $ \sigma ( D) = 
 \sigma ( D_x , D_\xi; D_y, D_\eta )  $ as usual,  
and 
\[
\Gamma_{\alpha, \beta, \rho,\gamma }(D) =( h^\alpha \pa_{(x,y)},
h^\beta \pa_{(\xi,\eta)}))^\rho \partial^{\gamma_1} 
\partial^{\gamma_2}.
\]
\end{lemma}

\begin{proof}
Following \cite[Lemma 3.6]{SjZw-frac} we recall that
$$
c(x,\xi) =  \exp (i h \sigma(D)/2) a(x,\xi) b(y,\eta)|_{x=y,\eta=\xi}
$$
and hence by Taylor's theorem the remainder may be expressed as
\begin{equation}\label{eN}
e_N(x,\xi) = \frac{1}{N!} \int_0^1 (1-t)^N  \exp(it h \sigma(D)/2) (ih \sigma(D)/2)^{N+1}\big(a(x,\xi)b(y,\eta)\big)\, dt\big\rvert_{x=y,\eta=\xi} .
\end{equation}
Likewise $\pa^\gamma e_N$ is a sum of terms of the form
\begin{equation}\label{eNderiv}
\text{const} \times \int_0^1 (1-t)^N \pa_{(x,\xi)}^{\gamma_1} \pa_{(y,\eta)}^{\gamma_2} \exp(it h \sigma(D)/2) (ih \sigma(D)/2)^{N+1}\big(a(x,\xi)b(y,\eta)\big)\, dt\big\rvert_{x=y,\eta=\xi} 
\end{equation}
where $\gamma_1+\gamma_2=\gamma.$
We further recall that for any non-degenerate real quadratic form $A$ there exists $M$ such that for all $f$,
$$
\abs{\pa^\gamma \exp(iA(D)/2) f} \leq C \sum_{\abs{\rho}<M} \sup
\abs{\pa^{\gamma+\rho} f}
$$
(where the $\sup$ is over all phase and base variables---in our case,
$(x,\xi,y,\eta)$).  Now we take $A(D)=th\sigma(D)$ and note that
$$
h\sigma(D) =\sigma(h^\alpha D_x, h^\beta D_\xi; h^\alpha D_y, h^\beta D_\eta),
$$
Rescaling the $x,y$ variables by $h^\alpha$ and $\xi,\eta$ by
$h^\beta$ shows that we may estimate
$$
\abs{\pa^{\gamma_1}\pa^{\gamma_2} \exp(ith \sigma(D)/2) f} \leq C \sum_{\abs{\rho}<M} \sup
\abs{\pa^{\gamma_1}\pa^{\gamma_2} (h^\alpha\pa_{(x,y)},h^\beta\pa_{(\xi,\eta)})^\rho f}
$$
Thus we may estimate the integrand of of \eqref{eNderiv} uniformly by
a constant times
$$
 h^{N+1}\sup \sum_{\abs{\rho}<M} 
\abs{\pa^{\gamma_1} \pa^{\gamma_2} (h^\alpha\pa_{x,y},h^\beta\pa_{\xi,\eta})^\rho 
\sigma(D)^{N+1}a(x,\xi) b(y,\eta)},
$$
and the result follows.
\end{proof}

As a particular consequence we notice that if $ a \in 
\s_{\alpha,\beta} ( T^* \RR^n)  $ and $ b \in  \s ( T^* \RR^n ) $ then 
\begin{align}
\label{eq:abc}
  c ( x, \xi) = &
 \sum_{k=0}^N \frac{1}{k!} \left( 
i h \sigma ( D_x , D_\xi; D_y , D_\eta) \right)^k a ( x , \xi) 
b( y , \eta) |_{ x = y , \xi = \eta} \\
& + {\mathcal O}_{\s_{\alpha,\beta}}
( h^{N+1} \max \{ (\tilde{h}/h)^{(N+1) \alpha}, (\tilde{h}/h)^{(N+1) \beta} \} 
\,. \notag
\end{align}

We will let $\B$ denote the ``blowdown map''\footnote{We
  remark that introducing the coordinates $(X,\Xi)$ is tantamount to
  performing an anisotropic blowup centered at $x=\xi=h=0$.}
\begin{equation}\label{blowdown}
(x,\xi)=\B(X,\Xi)=((h/\th)^\alpha X, (h/\th)^\beta \Xi).
\end{equation}
The spaces of operators $\Psi_h$ and 
$\Psi_{\tilde{h}}$ are related via a unitary rescaling in the
following fashion.  
Let $a \in \s_{\alpha,\beta}^{k,m,\tilde{m}}$, and consider the
rescaled symbol
\be
a\left(\csh^{\alpha}X, \csh^{\beta} \Xi
\right)= a \circ \B \in \s_{0,0}^{k,m,\tilde{m}}.
\ee
Define the unitary operator $T_{h, \tilde{h}} u(X) = \csh^{\frac{n\alpha}{2}}u\left(
  \csh^{\alpha} X \right)$,   
so that
\be\label{rescaledquantization}
\Op_{\tilde{h}}^w(a\circ B) T_{h, \tilde{h}} u= T_{h, \tilde{h}} \Op_h^w(a) u.
\ee

\subsection{Proof of Lemma \ref{L:ml-inv}}

By virtue of the cutoff $\phi^w$, to begin we are
working microlocally in $\{|(x,\xi)| \leq \epsilon \}$.  We observe
that since $2m/(m+1)<2$, if we can show the estimate \eqref{E:ml-inv}
for $Q_1 = \tQ - h^2 V_1$, the estimate follows also for $\tQ$.  
Let
\[
q_1 = \xi^2 + A^{-2} -z 
\]
be the principal symbol of $Q_1$.  The function $A^{-2} = (1 +
x^{2m})^{-1/m}$ is analytic near $x = 0$, and since $|x| \leq \epsilon$ is
small, we expand $A^{-2}$ in a Taylor series about $x = 0$ and write
\[
q_1 = \xi^2 - \frac{1}{m} x^{2m}(1 + a(x)) -z_1,
\]
where $z_1 = z - 1 \in [-\epsilon, \epsilon],$ and $a(x) = \O(x^{2m}).$

The Hamilton vector field $\hamvf$ associated to the symbol $q_1$ is given by
$$
\hamvf = 2\xi\pa_{x} +(2 x^{2m-1}+ \O(x^{4m-1})) \pa_{\xi}.
$$
We will consider a commutant localizing in this region and singular at the
origin in a controlled way: as above we introduce new variables
$$
\Xi=\frac{\xi}{(h/\th)^{m\alpha}},\quad X = \frac{x}{(h/\th)^\alpha},
$$
with $$\alpha = \frac 1{m+1}.$$ (When we wish to be more precise
below, we will explicitly use the map $(x,\xi) =\B(X,\Xi)$ in this
coordinate change; for the moment, we simply abuse notation.)  As
$m\alpha +\alpha = 1,$ we note that quantizations of symbolic
functions of $X,\Xi$ lie in the pseudodifferential calculus, hence the
symbol of the composition of two such operators depends
\emph{globally} on the symbols of the two operators.  It is in order
to cope with this issue that we employ the two parameter calculus.

We remark that in the new ``blown-up'' coordinates $\Xi,X,$
\begin{equation}\label{blownupvf}
\hamvf= (h/\th)^{\frac{m-1}{m+1}}\big(\Xi \pa_X+
X^{2m-1}\pa_\Xi+\O((h/\th)^{2m\alpha} X^{2m})\pa_\Xi\big)
\end{equation}

Now fix a small $\ep_0>0$ and set 
$$
\Lambda(s) = \int_0^s \ang{s'}^{-1-\ep_0} \, ds';
$$
$\Lambda$ is of course a symbol of order $0,$ with $\Lambda(s)
\sim s$ near $s=0.$

We introduce the singular symbol
$$
a(x,\xi;h) = \Lambda(\Xi)\Lambda(X)\chi(x)\chi(\xi)= \Lambda(\xi/(h/\th)^{m\alpha}) \Lambda(x/(h/\th)^\alpha)\chi(x)\chi(\xi),
$$
where $\chi(s)$ is a cutoff function equal to $1$ for $\abs{s}<\delta_1$
and $0$ for $s>2\delta_1$ ($\delta_1$ will be chosen shortly).
Then $a$ is bounded, and a $0$ symbol in $X,\Xi:$
$$
\abs{\pa_X^\alpha \pa_\Xi^\beta a}\leq C_{\alpha,\beta}.
$$
(Recall that $x=(h/\th)^\alpha X$ and $\xi=(h/\th)^{m\alpha}\Xi.$)
Using \eqref{blownupvf}, it is simple to
compute
\begin{equation}\label{gdefn}
\begin{aligned}
\hamvf (a) = & (h/\th)^{\frac{m-1}{m+1}}\chi(x)\chi(\xi)\big(
\Lambda(\Xi)
\ang{X}^{-1-\ep_0}\Xi \\
& +X^{2m-1}\ang{\Xi}^{-1-\ep_0}
\Lambda(X) (1+\O(x^{2m}))\big)+r\\
= & 
(h/\th)^{\frac{m-1}{m+1}}\chi(x)\chi(\xi)\bigg(
 (h/\th)^{-m\alpha}\xi \Lambda(\xi/(h/\th)^{m\alpha})
\ang{x/(h/\th)^\alpha}^{-1-\ep_0}\\ &  +(h/\th)^{-(2m+1)\alpha} x^{2m-1}
\Lambda(x/(h/\th)^\alpha) \ang{\xi/(h/\th)^{m\alpha}}^{-1-\ep_0}
(1+\O(x^{2m}))\bigg)+r\\
\equiv & (h/\th)^{\frac{m-1}{m+1}} g+r
\end{aligned}
\end{equation}
with $$\supp r\subset \{\abs{x}>\delta_1\} \cup \{\abs{\xi}>\delta_1\}$$
($r$ comes from terms involving derivatives of $\chi(x)\chi(\xi)$).
Note that near $X=\Xi=0,$ since $\Lambda(s)\sim s$ for $s\sim 0,$ the term
\begin{equation}\label{g}
g=\Lambda(\Xi)
\ang{X}^{-1-\ep_0}\Xi +\ang{\Xi}^{-1-\ep_0}
\Lambda(X) X^{2m-1} (1+\O(x^{2m}))
\end{equation}
in $\hamvf(a)$ is bounded below by a multiple of
$\Xi^2+X^{2m}.$  Provided $\delta_1$ is chosen small enough (so we can
absorb the $\O(x^{2m})$ error term),  $g$ is in fact strictly positive
away from $X=\Xi=0,$ while in the region $\abs{(X,\Xi)}\geq 1,$ we find that
since $\sgn \Lambda(s)=\sgn (s),$ when $\abs{\Xi}\geq
\max(\abs{X}^{1+\ep_0}, 1)$
then
$$
g\geq \Lambda(\Xi)
\ang{X}^{-1-\ep_0}\Xi\gtrsim \frac{\abs{\Xi}}{\ang{\Xi}}\geq C>0,
$$
while for $\abs{X}^{1+\ep_0}\geq \max(\abs{\Xi},1),$ we have (providing
$\delta_1\ll 1$)
$$
g \geq (1/2) \ang{\Xi}^{-1-\ep_0} \Lambda(X) X^{2m-1} \gtrsim
\abs{X}^{-2(1+\ep_0)}\abs{X}^{2m-1} \geq C>0,
$$
provided $2(1+\ep_0)<2m-1.$  Thus, since the larger of $\abs{\Xi}$ and
$\abs{X}^{1+\ep_0}$ is assuredly greater than $1$ in the region of
interest, we have in fact shown that
$$
g \geq C>0\quad \text{ in } \{\Xi^2+X^2>1\}.
$$
Thus, we find
$$
\hamvf(a)=(h/\th)^{\frac{m-1}{m+1}}g+r
$$
with 
$$r = \O_{\s_{\alpha, \beta}}((h/\th)^{(m-1)/(m+1)}( (h/\th)^\alpha |
\Xi| + (h/\th)^\beta | X^{2m-1}  |)$$
 supported as above and
$$
g(X,\Xi;h) =  \begin{cases} c (\Xi^2 + X^{2m}) (1 + r_2), & \Xi^2 +X^2\leq 1\\
b, &  \Xi^2 +X^2\geq 1,
\end{cases}
$$
where $c >0$ is a constant,$r_2 = \O_{\s_{\alpha, \beta}}( \delta_1)$, and $b>0$ is elliptic.

We will require a positivity result dealing with operators satisfying
estimates of this type.% \jw{Put the $r_2$ term into this lemma, as I
  % didn't see how to easily deal with it later.}
\begin{lemma}\label{lemma:positivity0}
Let a real-valued symbol $\tg(x,\xi;h)$ satisfy
$$
\tg(x,\xi;h) =  \begin{cases} c (\xi^2 + x^{2m}) (1 + r_2), & \xi^2 +x^2\leq 1\\
b, &  \xi^2 +x^2\geq 1,
\end{cases}
$$
where $c >0$ is constant, $r_2 = \O_{\s_{\alpha, \beta}}( \delta_1),$ and $b>0$ is elliptic.
Then there exists $c_0>0$ such that
$$
\lll\Op_h^w(\tg)u, u \rrr \geq c_0h^{2m/(m+1)} \| u \|^2
$$
for $h$ sufficiently small.
\end{lemma}

\begin{proof}% \jw{This got a little more complicated; please doublecheck.}
Since $b>0$ is elliptic, there exists $\sigma >0$ sufficiently small
and independent of $h>0$ 
so that if $\lll \Op_h^w(\tg) u, u \rrr \leq \sigma \| u \|^2$, then $u$
has semiclassical wavefront set contained in the set $\{ | x |^2 + |
\xi |^2 \leq 1/2 \}$.  
On this set, we may write
$$
\tg=(\xi^2+x^{2m})K^2
$$
with $K$ a strictly positive symbol.  The Weyl quantization has the
convenient feature that we thus have
$$
\Op_h^w(\tg) = \Op_h^w(K)^* (h^2D_x^2+ x^{2m}) \Op_h^w(K)+\O(h^2),
$$
and $\Op_h^w(K)\geq \ep_1>0.$
Then for $u$ microsupported in $\{ | x |^2 + |
\xi |^2 \leq 1/2 \}$ we thus compute
\[
\lll \Op_h^w(\tg) u, u \rrr \geq \ep_1\lll \Op_h^w(\xi^2 + x^{2m} ) u, u \rrr -
\O(h^2) \| u \|^2.
\]
The lower bound follows, for $h>0$ sufficiently small, from
Lemma \ref{L:Q-lower-bound}.

\end{proof}

We now employ this result to estimate $\Op_h^w (\hamvf(a)).$
\begin{lemma}\label{lemma:positivity1}
For $\tilde{h}>0$ sufficiently small, there exists $c>0$ such that
$\Op_h^w(g)>c\th^{2m/(m+1)},$ uniformly as $h \downarrow 0,$ where $g$
is given by \eqref{g}.
\end{lemma}
\begin{proof}
Note that we have written $g$ as a function of $X,\Xi,$ so in changing
variables to $x,\xi$ we are tacitly employing the blowdown map $\B.$
In particular, we are interested in estimating $\Op_h^w(g\circ
\B^{-1})$ from below.
By \eqref{rescaledquantization},
$$
\Op_{\tilde{h}}^w(g) T_{h, \tilde{h}} u=T_{h, \tilde{h}} \Op_h^w(g
\circ \B^{-1}) u,
$$
hence
$$
\lll \Op_h^w(g \circ \B^{-1}) u, u \rrr = \lll T_{h, \tilde{h}}\Op_{\tilde{h}}^w(g)
T_{h, \tilde{h}} u, u \rrr \geq c \tilde{h}^{2m/(m+1)} \| u \|^2
$$
for $\tilde{h}$ sufficiently small, by unitarity of $T_{h,\tilde{h}}$
and Lemma~\ref{lemma:positivity0}, with $\tilde{h}$ replacing $h.$
This establishes the Lemma.
\end{proof}

Before completing the proof of Lemma \ref{L:ml-inv}, we need the
following lemma about the lower order terms in the expansion of the
commutator of $Q_1$ and $a^w$.  

\begin{lemma}
\label{L:Q-comm-error}
The symbol expansion of $[Q_1, a^w]$ in the $h$-Weyl calculus is of
the form
\begin{align*}
[Q_1, a^w] = & \Op_h^w \Bigg( \Big( 
\frac{i h}{2} \sigma ( D_x , D_\xi; D_y , D_\eta) \Big) (q_1(x, \xi)
a(y, \eta) - q_1(y, \eta) a ( x , \xi) ) |_{ x = y , \xi = \eta} \\
& + e (
x, \xi ) + r_3(x, \xi)\Bigg) ,
\end{align*}
where $e$ satisfies
\[
\Op_h^w(e) \leq C \tilde{h}^{-(m-3)/(m+1)} h^{2m/(m+1)} \Op_h^w(g),
\]
with $g$ given by \eqref{g} and $r_3$ supported in $\{ | (x, \xi) | \geq \delta_1 \}$.

\end{lemma}

\begin{proof}
Since everything is in the Weyl calculus, only the odd terms in the
exponential composition expansion are non-zero.  Hence the $h^2$ term
is zero in the Weyl expansion.  % Hence the first error
% term is
% \[
% e_3 = c_3 \sigma(D)^3 (q_1(x, \xi)
% a(y, \eta) - q_1(y, \eta) a ( x , \xi) )|_{ x = y , \xi = \eta}.
% \]
Now according to Lemma \ref{l:err} and the standard $L^2$ continuity
theorem for $h$-pseudodifferential operators, we need to estimate a
finite number of 
derivatives of the error:
\[
 | \partial^{\gamma} e_2 | \leq C h^{3}
\sum_{ \gamma_1 + \gamma_2 = \gamma } 
 \sup_{ 
{{( x, \xi) \in T^* \RR }
\atop{ ( y , \eta) \in T^* \RR }}} \sup_{
|\rho | \leq M  \,, \rho \in \NN^{4} }
\left|
\Gamma_{\alpha, \beta, \rho, \gamma}(D)
( \sigma ( D) ) ^{3} q_1 ( x , \xi)  
a ( y, \eta ) 
\right|  % \big |_{ x = y , \xi = \eta}
.
\]
However, since $q_1(x, \xi) = \xi^2 - x^{2m}(1 + a(x))$, we have
\[
D_x D_\xi q_1 = D_{\xi}^3 q_1 = 0,
\]
so that
\begin{align*}
\sigma(D)^3 & q_1(x, \xi) a(y, \eta) |_{ x = y , \xi = \eta} \\
& = D_x^3 q_1 D_\eta^3 a |_{ x = y , \xi = \eta} \\
& = c x^{2m-3}(1 + \O(x^{2m})) (\th/h)^{3m/(m+1)} \Lambda'''
((\th/h)^{m/(m+1)} \eta) \\
& \quad \times \Lambda((\th/h)^{1/(m+1)} y) \chi(y)
\chi(\eta)% \big |_{ x = y , \xi = \eta}
+ r_3,
\end{align*}
where $r_3$ is supported in $\{ | (x, \xi) | \geq \delta_1 \}$.  
Owing to the cutoffs $\chi(y) \chi(\eta)$ in the definition of $a$
(and the corresponding implicit cutoffs in $q_1$), we only need to
estimate this error in compact sets.  The derivatives
$h^{\beta} \partial_\eta$ and $h^\alpha \partial_y$
preserve the order of $e_2$ in $h$ and increase the order in $\th$, while the other derivatives lead
to higher powers in $h/\th$ in the symbol expansion.  Hence we need only estimate $e_2$,
as the derivatives satisfy similar estimates.

In order to estimate $e_2$, we again use conjugation to
the $2$-parameter calculus.  We have
\[
\| \Op_h^w(e_2) u \| = \| T_{h, \th} \Op_h^w(e_2) T_{h, \th}^{-1}
T_{h, \th} u \| \leq \| T_{h, \th} \Op_h^w(e_2) T_{h, \th}^{-1}
\|_{L^2 \to L^2} \| u \|,
\]
by unitarity of $T_{h, \th}$.  But $T_{h, \th} \Op_h^w(e_2) T_{h,
  \th}^{-1} = \Op_{\th}^w(e_2 \circ \B )$ and
\begin{align*}
e_2 \circ \B & = h^3 (h/\th)^{(2m-3)\alpha} X^{2m-3}(1 + \O(x^{2m})) (\th/h)^{3m/(m+1)} \Lambda'''
(\Xi) \\
& \quad \times \Lambda(X) \chi(x)
\chi(\xi) + r_3 \circ \B ,
\end{align*}
and we may estimate the first term above by
$$
C h^{2m/(m+1)} \th^{(m+3)/(m+1)}X^{2m-3} \Lambda'''(\Xi) \chi(x)
\chi(\xi),
$$
which in turn is bounded above by
\begin{equation}\label{gkineq}
\begin{cases}
C h^{2m/(m+1)} \th^{(m+3)/(m+1)} , \,\, |X| \leq 1, \\
C h^{2m/(m+1)} \th^{(m+3)/(m+1)} g, \,\, | X | \geq 1.
\end{cases}
\end{equation}
It now suffices to verify that for 
$$
k=X^{2m-3} \Lambda'''(\Xi) \chi(x)
\chi(\xi),
$$
$$
\Op_{h}^w(k\circ \B^{-1}) \leq C \th^{-2m/(m+1)}\Op_{h}^w (g\circ \B^{-1}),
$$
i.e., that for all $u(X),$
$$
\ang{\Op_{h}^w(k\circ \B^{-1})u,u} \leq C
\th^{-2m/(m+1)}\ang{\Op_{h}^w (g\circ \B^{-1})u,u}.
$$
We now rescale and return to performing $\th$ quantization in the $X$
variable.  For $u(X)$ microsupported away from the origin in
$(X,\Xi),$ the desired estimate follows from the second inequality in
\eqref{gkineq} (indeed the $\th^{-2m/(m+1)}$ factor on the RHS may be
omitted), while for $u$ microsupported near the origin, it follows
from the lower bound of Lemma~\ref{lemma:positivity0}.
%An application of the Fefferman-Phong inequality finishes the proof.
\end{proof}

We are now able to prove the resolvent estimate Lemma \ref{L:ml-inv}.
Let $v=\varphi^w u,$ with $\varphi$ chosen to have support inside the
set where $\chi(x)\chi(\xi)=1;$ thus the terms $r$ and $r_3$ above are supported
away from the support of $\varphi.$  Then
Lemmas \ref{lemma:positivity1} and \ref{L:Q-comm-error} yield
\begin{align*}
i\ang{[Q_1-z,a^w]v,v}&=h\ang{\Op_h^w(\hamvf(a))v,v}+\ang{\Op_h^w(e_2)u,u}
\\
&= h (h/\th)^{(m-1)/(m+1)} \ang{\Op_h^w(g)v,v}+\ang{\Op_h^w(e_2)u,u}
\\
&= h^{2m/(m+1)}\big( \th^{-(m-1)/(m+1)}+\O(\th^{-(m-3)/(m+1)})\big)\ang{\Op_h^w(g)v,v}
\\
&\geq C
h^{2m/(m+1)} \th \norm{v}^2,
\end{align*}
 for
$\th$ sufficiently small.
On the other hand, we certainly have
$$
\big\lvert \ang{[Q_1-z,a^w]v,v}\big\rvert \leq C \norm{(Q_1-z)v}\norm{v},
$$
hence the desired bound follows once we fix $\th>0$. \qed

\section{Resonances and Quasimodes}

In this section, we construct quasimodes for the model
operator near $(0,0)$ in phase space.  
Let 
\[
\tP = -h^2 \partial_x^2 - m^{-1} x^{2m}
\]
locally near $x = 0$.  We will construct quasimodes which are
localized very close to $x = 0$, so this should be a decent
approximation.

%  It is well-known (see the appendix) that the operator
% \[
% \tilde{Q} = -\partial_x^2 + x^{2m}
% \]
% has a unique ground state $\tilde{Q} v_0 = \lambda_0 v_0$, with
% $\lambda_0>0$, and $v_0$ is a Schwartz class function, and moreover is
% real-analytic.  Then, 

Complex scaling $(x, \xi ) \mapsto (e^{i \pi/(2m+2)} x, e^{-i \pi /(2m+2)}
\xi)$ sends $\tP$ to a multiple of the quantum anharmonic oscillator.  
As in
the appendix, we find there is a Schwartz class function $v(x) = v_0 ( x
h^{-1/(m+1)})$ which is an un-normalized ground state for the equation
\[
(-h^2 \partial_x^2 + m^{-1} x^{2m} ) v = h^{2m/(m+1)} \lambda_0 v.
\]
This suggests there are resonances for the operator
$\tP$ with imaginary part to leading order $c_0
h^{2m/(m+1)}$, although this is only a heuristic.  We use a complex WKB approximation to get an explicit
formula for a localized approximate resonant state.

Let $E = (\alpha + i \beta)h^{2m/(m+1)} $, $\alpha, \beta>0$
independent of $h$.  Let the phase function % \jw{Changed name of phase
  % function to avoid conflict with $\phi$ in subsequent section.}
\[
\varpi(x) = \int_0^x (E + m^{-1} y^{2m})^{1/2} dy,
\]
where the branch of the square root is chosen to have positive
imaginary part.  Let 
\[
u(x) = (\varpi')^{-1/2} e^{i \varpi  / h},
\]
so that
\[
(hD)^2 u = (\varpi')^2 u + f u,
\]
where
\begin{align*}
f & = (\varpi')^{1/2} (hD)^2 (\varpi')^{-1/2} \\
& = -h^2 \left( \frac{3}{4} (\varpi')^{-2} (\varpi'')^2 - \frac{1}{2}
  (\varpi')^{-1} \varpi ''' \right).
\end{align*}

\begin{lemma}
The phase function $\varpi$ satisfies the following properties:
\begin{description}

\item[(i)]  There exists $C>0$ independent of $h$ such that 
\[
| \Im \varpi | \leq C\begin{cases} h(1 + \log(x/h^{1/2} )), \quad m =
  1, \\
 h, \quad m \geq 2.
\end{cases}
\]
In particular, if $| x | \leq C h^{1/(m+1)}$, $| \Im \varpi| \leq C'$
for some $C'>0$ independent of $h$.

\item[(ii)]  There exists $C>0$ independent of $h$ such that 
\[
C^{-1} \sqrt{ h^{2m/(m+1)} + m^{-1}  x^{2m} } \leq | \varpi'(x) | \leq C \sqrt{
  h^{2m/(m+1)} + m^{-1}  x^{2m} }
\]

\item[(iii)]
\[
\begin{cases}
\varpi' = (E + m^{-1} x^{2m})^{1/2}, \\
\varpi'' = x^{2m-1} (\varpi')^{-1}, \\
\varpi''' = \left( (1-1/m)  x^{4m-2} + E  (2m-1) x^{2m-2}
\right) ( \varpi')^{-3}, 
\end{cases}
\]
In particular,
\[
f = -h^2 x^{2m-2} \left( \left( \frac{1}{4}  + \frac{1}{2m} 
  \right) x^{2m} -  \left( m - \frac{1}{2} \right) E \right) (\varpi'
)^{-4}.
\]

\end{description}

\end{lemma}

\begin{proof}
For (i) we write $\varpi' = s + it$ for $s$ and $t$ real valued, and then
\[
E + m^{-1} x^{2m} = s^2 - t^2 + 2 i st.
\]
Hence
\[
s^2 \geq s^2 - t^2 = \alpha h^{2m/(m+1)} + m^{-1}  x^{2m} ,
\]
so that
\[
t = \frac{\beta h^{2m/(m+1)}}{2s} \leq \frac{\beta
  h^{2m/(m+1)}}{2\sqrt{h^{2m/(m+1)} \alpha + m^{-1}  x^{2m}}}.
\]
Then
\begin{align*}
| \Im \varpi (x) | & \leq \int_0^{|x|} | \Im \varpi'(y)| dy \\
& \leq C \int_0^{h^{1/(m+1)}} h^{m/(m+1)} dy + C \int_{h^{1/(m+1)}}^x
h^{2m/(m+1)} y^{-m} dy \\
& = \begin{cases} \O ( h(1 + \log (x/h^{1/2}))), \quad m = 1, \\
\O(h), \quad m >1. \end{cases}
\end{align*}

Parts (ii) and (iii) are simple computations.

\end{proof}

In light of this lemma, $| u (x) |$ is comparable to $| \varpi'
|^{-1/2}$ for all $x$ for $m \geq 2$, and provided $| x | \leq C h^{1/2}$ when
$m=1$.  We are only interested in sharply localized quasimodes, so let
$$\gamma = h^{1/(m+1)},$$ choose $\chi(s) \in \Ci_c( \reals)$ such that
$\chi \equiv 1$ for $| s | \leq 1$ and $\supp \chi \subset [-2,2]$.
Let
\[
\tu(x) = \chi(x/\gamma) u(x),
\]
and, since $| \varpi'(x) | \sim h^{m/(m+1)}$ for $| x | \leq 2
h^{1/(m+1)}$, we compute: 
\begin{align*}
\| \tu \|_{L^2}^2 & = \int_{| x | \leq 2 \gamma} \chi(x/\gamma)^2 | u
|^2 dx \\
& \sim \int_{|x| \leq 2\gamma} \chi(x/\gamma)^2 | \varpi' |^{-1} dx \\
& \sim h^{1/(m+1)} h^{ -m/(m+1)} \\
& \sim h^{(1-m)/(1+m)}.
\end{align*}

Further, $\tu$ satisfies the following equation:
\begin{align*}
(hD)^2 \tu & = \chi(x/\gamma) (hD)^2 u + [(hD)^2, \chi(x/\gamma)] u \\
& = (\varpi')^2 \tu + f \tu + [(hD)^2, \chi(x/\gamma)] u \\
& = (\varpi')^2 \tu + R,
\end{align*}
where 
\[
R = f \tu + [(hD)^2, \chi(x/\gamma)] u.
\]

\begin{lemma}
The remainder $R$ satisfies
\begin{equation}
\label{E:R-remainder}
\| R \|_{L^2} = \O (h^{2m/(m+1)}) \| \tu \|_{L^2}.
\end{equation}

\end{lemma}

\begin{proof}
We have already computed the function $f$, which is readily seen to
satisfy 
\[
\| f \|_{L^\infty(\supp (\tu ))} = \O(h^{2m/(m+1)}),
\]
since $\supp (\tu) \subset \{ | x | \leq 2 h^{1/(m+1)}$.

On the other hand, since $\| \tu \|_{L^2} \sim h^{(1-m)/2(1+m)}$, we need only show that
\[
\| [(hD)^2, \chi(x/\gamma)] u\|_{L^2} \leq C h^{(3m+1)/2(m+1)}.
\]
We compute:
\begin{align*}
[(hD)^2, \chi(x/\gamma)] u & = -h^2 \gamma^{-2} \chi'' u + 2\frac{h}{i}
\gamma^{-1} \chi' hD u \\
& = -h^2 \gamma^{-2} \chi'' u + 2\frac{h}{i}
\gamma^{-1} \chi'  \left(-\frac{h}{2i} \frac{\varpi''}{\varpi'} + \varpi'
\right) u \\
& = -h^2 \gamma^{-2} \chi'' u + 2\frac{h}{i}
\gamma^{-1} \chi' \left( -\frac{h}{2i} \frac{ x^{2m-1}}{(\varpi')^2} +
  \varpi' \right) u.
\end{align*}
The first term is estimated:
\[
\| h^2 \gamma^{-2} \chi'' u \|_{L^2} = \O(h^{2m/(m+1)}) \| u
\|_{L^2(\supp ( \tu))} =\O(h^{(3m+1)/2(m+1)}).
\]
Similarly, the remaining two terms are estimated:
\begin{align*}
& \left\| 2\frac{h}{i}
\gamma^{-1} \chi' \left( -\frac{h}{2i} \frac{ x^{2m-1}}{(\varpi')^2} +
  \varpi' \right) u \right\|_{L^2} \\
 & \quad = \O(h^{m/(m+1)} h^1 h^{(2m-1)/(m+1)}
h^{-2m/(m+1)}) \| u
\|_{L^2(\supp ( \tu))}  \\
& \quad \quad + \O(h^{m/(m+1)} h^{2m/(m+1)} )  \| u
\|_{L^2(\supp ( \tu))}  \\
& \quad = \O(h^{(3m+1)/2(m+1)}).
\end{align*}

\end{proof}

\subsection{Sharp local smoothing of quasimodes}

In this subsection, we show that the quasimode $\tu$ constructed above
can be used to 
saturate the local smoothing estimate of Theorem \ref{T:smoothing}.
We observe that due to the estimate \eqref{E:est-away-0}, we have
perfect local smoothing in the ``radial'' direction $x$.  Hence we
only consider $\theta$ regularity.

\begin{theorem}
\label{T:sharp}

Let $\phi_0(x, \theta) = e^{ik \theta} \tu(x)$, where $\tu \in \Ci_c (
\reals)$ was constructed in the previous section, with $h=\abs{k}^{-1},$ where $|k|$ is
taken sufficiently large, and the parameter $m \geq 2$ as usual.  Suppose $\psi$
solves
\[
\begin{cases}
(D_t + \tDelta) \psi = 0, \\
\psi|_{t=0} = \phi_0.
\end{cases}
\]
Then for any $\chi
\in \Ci_c( \reals)$ such that $\chi \equiv 1$ on $\supp \tu$, and
$A>0$ sufficiently large, independent of $k$, there exists a constant
$C_0>0$ independent of $k$ such that 
\begin{equation}\label{saturated}
\int_0^{|k|^{-2/(m+1)}/A} \| \lll D_\theta \rrr \chi \psi \|_{L^2}^2
dt \geq C_0^{-1} \| \lll D_\theta \rrr^{m/(m+1)} \phi_0 \|_{L^2}^2.
\end{equation}

\end{theorem}

\begin{remark}
  The theorem states that on the {\it weak semiclassical} time scale
  $| t | \lesssim |k|^{-2/(m+1)}$, the
  local smoothing estimate in this paper is sharp.  Evidently this
  implies that on any fixed finite time scale the theorem is sharp;
  the theorem stated above gives more information.  That is, it
  demonstrates that even on a semiclassical time scale, the local
  smoothing estimate really cannot be improved.

In addition, as the proof will indicate, no weight $\chi$ is
necessary, because on the semiclassical time scale we have essentially
finite propagation speed.

\end{remark}

\begin{remark}
The analogue of Theorem \ref{T:sharp} when the parameter $m = 1$,
where a log loss is expected, is contained in the work of
Bony-Burq-Ramond \cite{BBT-lower}.

\end{remark}

\begin{proof}

The technique of proof is to simply evolve the stationary quasimode as
if the equation were separated.  The advantage is that this separated
stationary ``solution'' remains compactly supported for all time.  Of
course this is not an exact solution, and generates an
inhomogeneous error which must be estimated using energy estimates.
It is here that we use the semiclassical time scale.  Combining the
two estimates yields the theorem.

Let $\tu$ be as above for $h =
| k |^{-1}$ (after a suitable $L^2$ normalization), and let
\[
\phi_0 (x, \theta) = \tu e^{ik \theta},
\]
as in the statement of the theorem.  
Let $$\phi(t, x, \theta) = e^{it \tau} \phi_0$$ for some $\tau \in \cx$
to be determined.  Since the support of $\tu$ is very small, contained
in $\{ | x | \leq h^{1/(m+1)} / \gamma \}$, we have 
\[
A^{-2} = (1 + x^{2m} )^{-1/m} = 1 -\frac{1}{m} x^{2m} +
\O(h^{4m/(m+1)})
\]
on $\supp \tu$.  Then
\begin{align*}
(D_t + \tDelta) \phi & = (D_t + P_k )\phi \\
& = ( \tau - D_x^2 - A^{-2} k^2 - V_1(x) ) \phi \\
& = k^2 e^{it \tau} e^{i k \theta} \left[\left( \tau k^{-2} - (k^{-2}D_x^2 + 1 - \frac{1}{m}
  x^{2m} ) \right) \tu  +  \O( k^{-2}) \tu \right] \\
& = k^2 e^{it \tau} e^{i k \theta}  \left[ \left( \tau k^{-2} - 1 -
    E_0 \right) \tu + R + \O( k^{-2}) \tu \right],
\end{align*}% \jw{The factor $1/m$ is an annoying difference with the
  % model problem.  Should reconcile to fix this minor issue.}
where $R$ satisfies the remainder estimate \eqref{E:R-remainder},
i.e.,
$$
\|R\|_{L^2} = \O ( h^{2m/(m+1)}).
$$  Set
\[
\tau = k^2 + k^2 E_0 = k^2(1 + \alpha k^{-2m/(m+1)}) + i \beta
k^{2/(m+1)}, \,\,\, \alpha, \beta >0
\]
so that we have
\[
\begin{cases}
(D_t + \tDelta) \phi = \tR, \\
\phi(0, x, \theta) = \phi_0
\end{cases}
\]
with
\begin{equation}
\label{E:tR}
\tR = k^2 e^{i t \tau } e^{i k \theta} (R(x, k) + \O(k^{-2}) \tu ).
\end{equation}
Thus, we obtain
\begin{equation}\label{tRestimates}
\|\tR\| \leq C k^2 \lvert e^{it\tau}\rvert k^{-2m/(m+1)} \|\tu\|
\end{equation}
and since on every function in question, $\lll D_\theta \rrr = \lll k
\rrr,$ we furthermore have
\begin{equation}\label{tRestimates2}
\|\lll D_\theta\rrr \tR\| \leq C k^2 \lvert e^{it\tau}\rvert k^{-2m/(m+1)} \|\lll D_\theta\rrr\phi_0\|
\end{equation}

We can readily verify that $\phi$ saturates the local smoothing
estimate of Theorem \ref{T:smoothing} {\it on any time scale}:
\begin{align}
\| \lll D_\theta \rrr \phi \|_{L^2([0,T])L^2}^2 & = \int_0^T \| e^{it \tau} \lll D_\theta \rrr\phi_0
\|^2_{L^2} dt \notag \\
& = \int_0^T e^{-2t\beta |k|^{2/(m+1)}} \| \lll D_\theta \rrr \phi_0 \|_{L^2}^2 dt \notag
\\
& = \frac{1 - e^{-2 T\beta | k |^{2/(m+1)}}}{2 \beta |k|^{2/(m+1)}} \|
\lll D_\theta \rrr \phi_0
\|_{L^2}^2 \notag \\
& = \frac{1 - e^{-2 T\beta | k |^{2/(m+1)}}}{ 2 \beta } \| |
D_\theta|^{-1/(m+1)} \lll D_\theta \rrr
\phi_0 \|_{L^2}^2 \notag \\
& = B \| |
D_\theta|^{-1/(m+1)} \lll D_\theta \rrr
\phi_0 \|_{L^2}^2,
 \label{E:sharp-smoothing}
\end{align}
where we let
\begin{equation}\label{Bdefn}
B= \frac{1 - e^{-2 T\beta | k |^{2/(m+1)}}}{ 2 \beta }.
\end{equation}
What we aim to do, then, is show that $\phi$ is close enough to a
solution of the Schr\"odinger equation that the same form of estimate
holds for the nearby solution with initial data $\phi_0.$

Thus, let $L(t)$ be the unitary Schr\"odinger propagator:
\[
\begin{cases}
(D_t + \tDelta) L = 0, \\
L(0) = \id,
\end{cases}
\]
and write using Duhamel's formula:
\[
\phi(t) = L(t) \phi_0 + i \int_0^t L(t) L^*(s) \tR(s) ds =: \phi_{\text{h}} + \phi_{\text{ih}},
\]
where $\phi_{\text{h}}$ and $\phi_{\text{ih}}$ are the homogeneous and
inhomogeneous parts respectively.  We want to show the 
homogeneous smoothing effect is saturated, for which we need to show
the 
inhomogeneous term is small and can be absorbed into the homogeneous term.  % Unfortunately for now, this does not
% seem sufficient, however we record it here anyway.  We compute:
% \begin{align*}
% \| \phi_{\text{ih}} \|_{L^2([0,T]) H^{1/(m+1)}} & = \left\| \int_0^t L(t)
%   L^*(s) \tR(s) ds \right\|_{L^2([0,T]) H^{1/(m+1)}} \\
% & \leq \int_0^T  \left\| L(t)
%   L^*(s) \tR(s)  \right\|_{L^2([0,T]) H^{1/(m+1)}} ds \\
% & \leq \int_0^T e^{-\beta t |k|^{2/(m+1)}} (\| L(t) L^*(s) k^2 e^{i k
%   \theta} R \|_{L^2([0,T]) H^{1/(m+1)}} \\
% & \quad + \| L(t) L^*(s)  e^{ik \theta}\tu \|_{L^2([0,T]) H^{1/(m+1)}}) ds \\
% & \leq C \frac{(1 - e^{-\beta t | k |^{2/(m+1)}})}{\beta | k
%   |^{2/(m+1)}} (\| k^2 e^{i k
%   \theta} R \|_{ L^2} + \|  e^{ik \theta} \tu \|_{L^2} ) \\
% & \leq C \frac{(1 - e^{-\beta t | k |^{2/(m+1)}})}{\beta | k
%   |^{2/(m+1)}} ( k^2 |k|^{-2m/(m+1)} \| e^{ik \theta} \tu \|_{L^2} +
% \| e^{ik \theta} \tu \|_{L^2} ) \\
% & \leq C' \| \phi_0 \|_{L^2},
% \end{align*}
% where we have used Minkowski's integral inequality in time, the
% homogeneous local smoothing effect for $L(t)$ (Theorem
% \ref{T:smoothing}), and the estimates on $R$ in
% \eqref{E:R-remainder}.  This estimate shows that the inhomogeneous
% part does obey the same local smoothing law as the homogeneous part,
% however we have no control over the constant.  As the constant in
% estimate \eqref{E:sharp-smoothing} is bounded by $1/2$ this does not
% allow us to get an estimate from below on the smoothing effect for
% $L(t)$ by itself.

For this, we use an energy estimate, and localize in time to a scale
depending on $k$.  That is, if $$E(t) = \| \lll D_\theta \rrr \phi_{\text{ih}} \|^2_{L^2},$$ we
have% \jw{Changed $\gamma$ to $\nu$ here et seq.\ to avoid conflict
  % with earlier use.}
\begin{align*}
E' & = 2 \Re \lll \lll D_\theta \rrr \partial_t \phi_{\text{ih}}, \lll
D_\theta \rrr \phi_{\text{ih}}
\rrr_{L^2} \\ 
& = 2 \Re \lll \lll D_\theta \rrr ( -i \tDelta \phi_{\text{ih}} + i
\tR) , \lll D_\theta \rrr\phi_{\text{ih}}
\rrr_{L^2} \\
& = 2 \Re \lll  i \lll D_\theta \rrr \tR, \lll D_\theta \rrr\phi_{\text{ih}}
\rrr_{L^2} \\
& \leq \nu \| \lll D_\theta \rrr \tR \|_{L^2}^2 + \nu^{-1} E,
\end{align*}
for $\nu>0$ to be determined, so that
\begin{align*}
E(t) & \leq e^{\nu^{-1}t} \int_0^t e^{-\nu^{-1} s} \left( \nu \|
\lll D_\theta \rrr \tR(s) \|_{L^2}^2 \right) ds \\
& \leq e^{\nu^{-1} t} \left( \nu \| \lll D_\theta \rrr \tR
  \|_{L^2([0,T]) L^2 }^2 \right) \\
& \leq C \nu e^{\nu^{-1}t} \left( \int_0^T | e^{2is \tau} | ds  \right)(k^2
k^{-2m/(m+1)})^2 \| \lll D_\theta \rrr \phi_0 \|_{L^2}^2 \\
& \leq C (\nu e^{\nu^{-1}t}) \left( \frac{ 1 - e^{-2 \beta |k|^{2/(m+1)} T}}{2
  \beta}  \right) \\
& \quad \times |k|^{-2/(m+1)} (k^2
k^{-2m/(m+1)})^2 \| \lll D_\theta \rrr \phi_0 \|_{L^2}^2.
\end{align*}
Here we have used the estimate \eqref{tRestimates2} in the middle inequality.
Integrating in $0 \leq t \leq T$, we get
\begin{align*}
\| \lll D_\theta \rrr \phi_{\text{ih}} \|^2_{L^2([0,T]) L^2} & \leq C
\nu \left( \int_0^T
e^{\nu^{-1}t} dt \right)  B |k|^{2/(m+1)} \| \lll D_\theta \rrr \phi_0 \|_{L^2}^2 \\
& \leq C\nu^2 (e^{\nu^{-1} T} - 1) B |k|^{2/(m+1)} \| \lll
D_\theta \rrr \phi_0
\|_{L^2}^2,
\end{align*}
where $B$ is given by \eqref{Bdefn}.

If $\nu = |k|^{-2/(m+1)}/A$ for some large $A>0$ independent of
$k$, and if we take $T = \nu$, we have
\begin{align}
\| \lll D_\theta \rrr \phi_{\text{ih}} \|^2_{L^2([0,T]) L^2} & \leq \frac{C}{A^2} (e-1) B
|k|^{-2/(m+1)} \| \lll D_\theta \rrr \phi_0
\|_{L^2}^2 \notag \\
& = c_0 B \| |D_\theta|^{-1/(m+1)}\lll D_\theta \rrr \phi_0 \|^2_{L^2}, \label{E:phi-ih}
\end{align}
where
\[
c_0 = \frac{C}{A^2} (e-1) \leq \frac{1}{4}
\]
if we take $A$ sufficiently large.

Now, since $T$ depends on $k$, so does $B$ in principle, but 
\[
B = \left( \frac{ 1 - e^{-2 \beta |k|^{2/(m+1)} T}}{2
  \beta} \right) = \left( \frac{ 1 - e^{-2 \beta /A}}{2
  \beta} \right) >0
\]
independent of $k$.  Hence, combining \eqref{E:sharp-smoothing} with
\eqref{E:phi-ih}, we have for this choice of $T$ and $A$, and $\chi$
as in the statement of the theorem, 
\begin{align*}
\| \lll D_\theta \rrr  \chi L(t) \phi_0 \|_{L^2([0,T]) L^2} & \geq \| \lll
D_\theta \rrr \chi \phi \|_{L^2([0,T]) L^2}
- \| \lll D_\theta \rrr \chi \phi_{\text{ih}} \|_{L^2([0,T]) L^2} \\
& \geq \| \lll
D_\theta \rrr  \phi \|_{L^2([0,T]) L^2}
- \| \lll D_\theta \rrr \phi_{\text{ih}} \|_{L^2([0,T]) L^2} \\
& \geq \frac{1}{2} B^{1/2} \| | D_\theta|^{-1/(m+1)} \lll D_\theta
\rrr \phi_0 \|_{L^2}
\end{align*}
(using $\supp \phi\subset \{\chi=1\}$), with $B>0$ independent
of $k$ as above.

Thus, $\psi=\phi_{\text{h}}$ satisfies \eqref{saturated}, since
$\phi$ satisfies the estimate, while $\phi_{\text{ih}}$ is small
enough to be absorbed.
\end{proof}

\section{Sharp resolvent estimates with loss}

We now  prove Theorem~\ref{T:resolvent}.

We begin with the sharpness.  By the $T
T^*$ argument (see, for example \cite{Chr-sch2},% \jw{Give reference?}
if a better resolvent estimate held true, then a better
local smoothing estimate would also hold true.  But Theorem \ref{T:sharp}
shows the local smoothing estimate in Theorem \ref{T:smoothing} is
sharp.  Hence no better polynomial rate of decay in the resolvent
estimate can hold true.

In order to show the estimate is true in the first place, we conjugate 
the Laplacian on $X$ and decompose in Fourier modes as usual
\[
(-\tDelta -\lambda^2) = \bigoplus_{k = - \infty}^\infty (L_k -
\lambda^2).
\]
We break this sum into two pieces where either $k^2 \leq \lambda^2/2$
or not.  If $k^2 \leq \lambda^2/2$, then we use $h = \lambda^{-1}$ as
our semiclassical parameter, while if $k^2 > \lambda^2 / 2$, we use $h
= |k|^{-1}$.  We get
\[
(-\tDelta - \lambda^2) = \bigoplus_{k^2 \leq \lambda^2/2} \lambda^2 (
\tL_k - z_k) \oplus \bigoplus_{k^2 > \lambda^2/2} k^2 (\tL_k - z_k ).
\]
Here
\[
\tL_k = \begin{cases}
-h^2 \partial_x^2 + \frac{k^2}{\lambda^2} A^{-2}(x) + h^2 V_1(x),
\text{ if } k^2 \leq \lambda^2/2, \\
-h^2 \partial_x^2 + A^{-2}(x) + h^2 V_1(x), \text{ if } k^2 >
\lambda^2 / 2,
\end{cases}
\]
and
\[
z_k = \begin{cases} 1, \text{ if } k^2 \leq \lambda^2/2, \\
k^{-2} \lambda^2 , \text{ if } k^2 > \lambda^2 /2.
\end{cases}
\]
In the case $k^2 \leq \lambda^2/2$, the energy level $z_k$ is 
non-trapping, so by Proposition~\ref{P:ml-res-est}  the operator $\tL_k$ 
satisfies the estimate
\[
\| \chi(\tL_k - z_k )^{-1} \chi \|_{L^2 \to L^2} \leq C h^{-1} = C
\lambda,
\]
with constants uniform as $|k|$ and $\lambda$ both go to infinity.  In
the case $k^2 > \lambda^2/2$, the operator $\tL_k$ satisfies the
estimate
\[
\| \chi (\tL_k - z_k )^{-1} \chi \|_{L^2 \to L^2} \leq C h^{-2m/(m+1)}
= C |k|^{2m/(m+1)},
\]
with constants independent of $\lambda$ as $|k| \to \infty$.  Hence by
orthogonality of the Fourier eigenspaces,
\begin{align*}
\| \chi R(\lambda) \chi \|_{L^2 \to L^2} & \leq \max \{ \max_{|k|^2 \leq
  \lambda^2/2} \lambda^{-2} \| \chi(\tL_k - z_k)^{-1} \chi \|, \\
& \quad \quad \quad \quad
\sup_{|k| > \lambda^2/2} |k|^{-2} \| \chi(\tL_k - z_k)^{-1} \chi \| \}
\\
& \leq \max \{ C \lambda^{-1}, C \sup_{|k|^2 > \lambda^2/2}
|k|^{-2/(m+1)} \} \\
& \leq C \lambda^{-2m/(m+1)}.\qed
\end{align*}

\appendix

\section{Eigenfunction properties}

In this section we recall some standard results on
eigenfunctions.

We consider the classical eigenfunction problem
\[
P u = \lambda u,
\]
where
\[
P = -\partial_x^2 + x^{2m},
\]
with $m \in \ZZ$, $m \geq 2$ and $\lambda \geq 0$.  

The following is a standard result in spectral theory (see, for
example, \cite{ReSi-IV}).% \jw{Dropped
  % proof here.  Maybe refer to Taylor v.2?  Don't have it with me right now.}
\begin{lemma}
\label{L:ao-ef-prop}
As an operator on $L^2$ with domain $\mathcal{S}(\RR),$ $P$ is
essentially self-adjoint.  It has pure point spectrum $\lambda_j \to
\infty$, $\lambda_0 >0$, and every eigenfunction is of Schwartz class.
\end{lemma}

We now remark the semiclassical anharmonic oscillator estimate, which
follows from Lemma~\ref{L:ao-ef-prop} via the rescaling
$X=h^{-1/(m+1)} x.$% \jw{Dropped detailed proof.}
\begin{lemma}
\label{L:Q-lower-bound}
Let $Q = -h^2 \partial_x^2 + x^{2m}$ for some $m \in \ZZ$, $m \geq
1$.  Then there exists $c = c_{m}>0$ such that for any $u \in L^2( \reals)$,
\[
\| Q u \| \geq c h^{2m/(m+1)} \| u \|.
\]
\end{lemma}

% \begin{proof}
% For $m = 1$ this is just the harmonic oscillator, for which the bound
% is well-known.  From Lemma \ref{L:ao-ef-prop} (see also, for example,
% \cite{Vor-eaqc}), we have for $m \geq 2$, the classical problem
% \[
% (-\partial_x^2 + x^{2m}) u = \lambda u
% \]
% is solved in $L^2$ for {\it positive} discrete values of $\lambda$.

% Let $\lambda_0>0$ be the smallest classical eigenvalue and rescale in
% phase space
% \[
% (x, \xi) \mapsto (h^{-1/(m+1)} x, h^{1/(m+1)} \xi)
% \]
% to get
% \[
% (-h^{2/(m+1)} \partial_x^2 + h^{-2m/(m+1)} x^{2m}) u = \lambda_0 u,
% \]
% which evidently implies
% \[
% (-h^{2} \partial_x^2 + x^{2m}) u = h^{2m/(m+1)} \lambda_0 u.
% \]
% \end{proof}

\bibliographystyle{alpha}
\bibliography{Wunsch-bib}

\end{document}